\documentclass[11pt]{amsart}
\usepackage[left=3.5cm, right=3.5cm, top=3cm, bottom=3cm]{geometry}
\usepackage{t1enc}
\usepackage[utf8]{inputenc}
\usepackage[english]{babel}
\usepackage{amsmath, amssymb, amsfonts, amsthm}
\usepackage{enumitem}
\usepackage{graphicx, color, calc}
\usepackage[colorlinks=true]{hyperref}
\usepackage{url}
\usepackage[capitalize, noabbrev]{cleveref}
\usepackage[msc-links, alphabetic]{amsrefs}
\usepackage{minted}
\usepackage{mathtools}
\usepackage{multirow}
\usepackage{tikz-cd}
\usepackage[tableposition=below]{caption}
\newtheorem{theorem}{Theorem}
\newtheorem{prop}[theorem]{Proposition}
\AddToHook{env/prop/begin}{\crefalias{theorem}{prop}}

\Crefname{prop}{Proposition}{Propositions}

\definecolor{bg}{rgb}{.95, .95, .95}
\setminted{bgcolor=bg}
\newcommand{\code}{\mintinline{py}}

\renewcommand{\S}{\Sigma}
\newcommand{\Z}{\mathbb Z}
\newcommand{\Q}{\mathbb Q}
\renewcommand{\d}{\partial}
\DeclareMathOperator{\Hom}{Hom}
\DeclareMathOperator{\Tor}{Tor}
\newcommand{\CKH}{\mathcal C \mathrm{Kh}}
\newcommand{\KH}{\mathrm{Kh}}
\newcommand{\sub}{\subseteq}

\graphicspath{{figures/}}

\newlist{myenumerate}{enumerate}{1}
\setlist[myenumerate]{label=\textup{(\alph*)}}

\newcommand{\HSimage}[1]{\raisebox{-.375\height}{\includegraphics[scale=.18,trim=0 .3cm 0 0]{#1}}\null}
\newcommand{\saddle}{\raisebox{-.1\height}{\includegraphics[scale=.2]{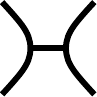}}}
\newcommand{\crossing}{\raisebox{-.1\height}{\includegraphics[scale=.2]{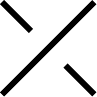}}}
\newcommand{\zerosmoothing}{\raisebox{-.1\height}{\includegraphics[scale=.2]{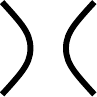}}}
\newcommand{\onesmoothing}{\raisebox{-.1\height}{\includegraphics[scale=.2]{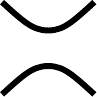}}}
\newcommand{\death}{\includegraphics[scale=.2,trim=0 .2cm 0 0]{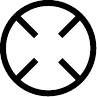}\null}
\newcommand{\birth}{\includegraphics[scale=.15,trim=0 .2cm 0 0]{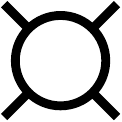}\null}

\newcommand{\kdot}{\includegraphics[scale=.2,trim=0 .2cm 0 0]{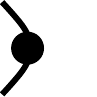}\null}
\newcommand{\kocirc}{\HSimage{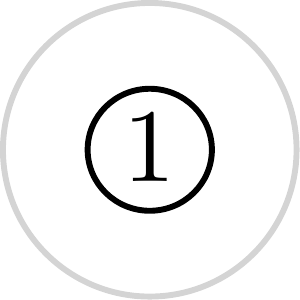}}
\newcommand{\kxcirc}{\HSimage{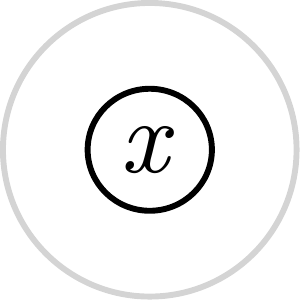}}
\newcommand{\kcobdo}{\HSimage{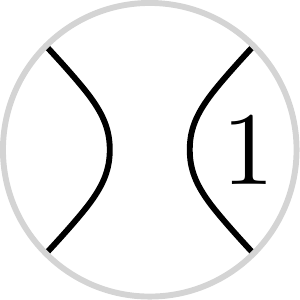}}
\newcommand{\kcobdx}{\HSimage{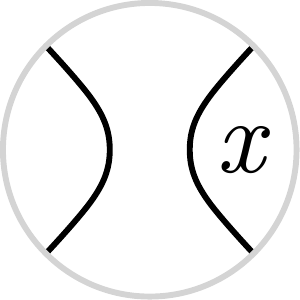}}
\newcommand{\kcobdoo}{\HSimage{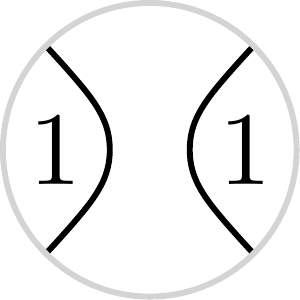}}
\newcommand{\kcobdox}{\HSimage{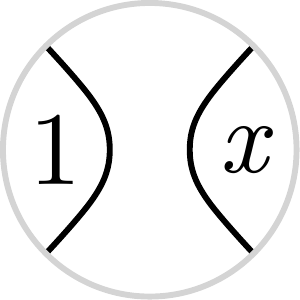}}
\newcommand{\kcobdxo}{\HSimage{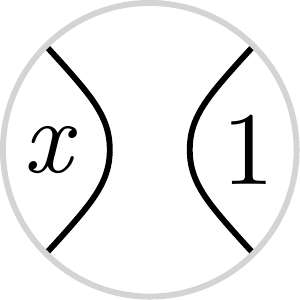}}
\newcommand{\kcobdxx}{\HSimage{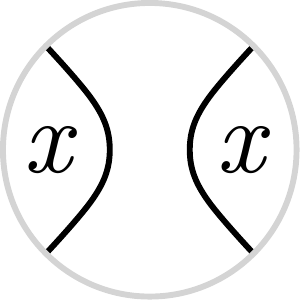}}
\newcommand{\kcobio}{\HSimage{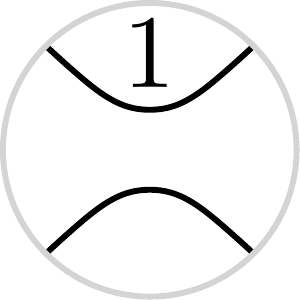}}
\newcommand{\kcobix}{\HSimage{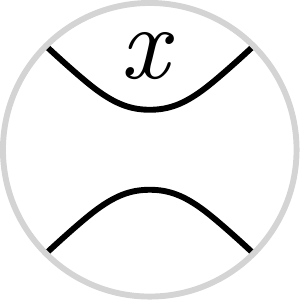}}

\newcommand{\kcobiox}{\HSimage{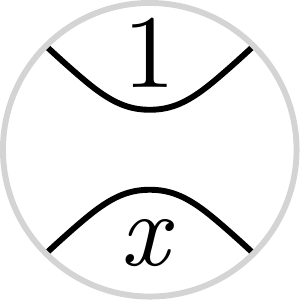}}
\newcommand{\kcobixo}{\HSimage{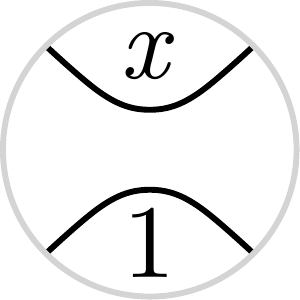}}
\newcommand{\kcobixx}{\HSimage{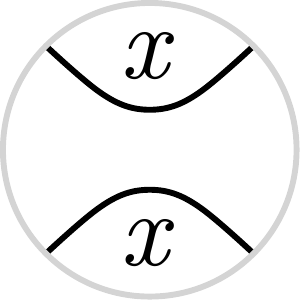}}

\newcommand{\kropos}{\HSimage{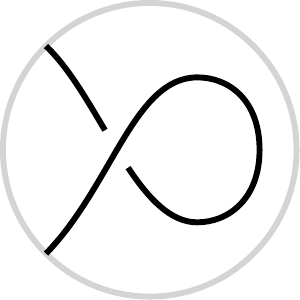}}
\newcommand{\kroneg}{\HSimage{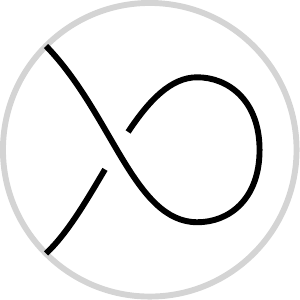}}
\newcommand{\kroarc}{\HSimage{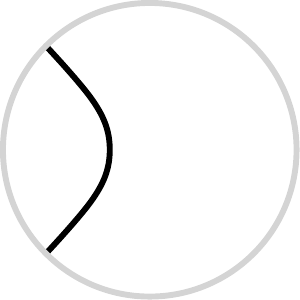}}
\newcommand{\krosa}{\HSimage{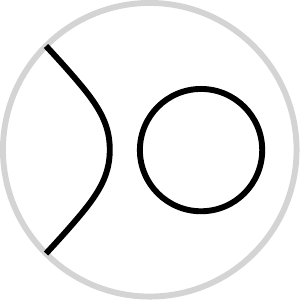}}
\newcommand{\krosb}{\HSimage{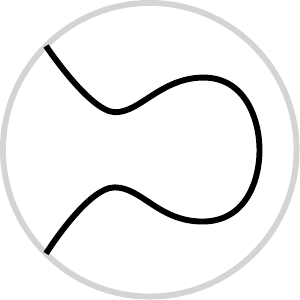}}
\newcommand{\kroma}{\HSimage{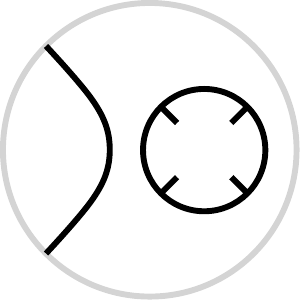}}
\newcommand{\kromb}{\HSimage{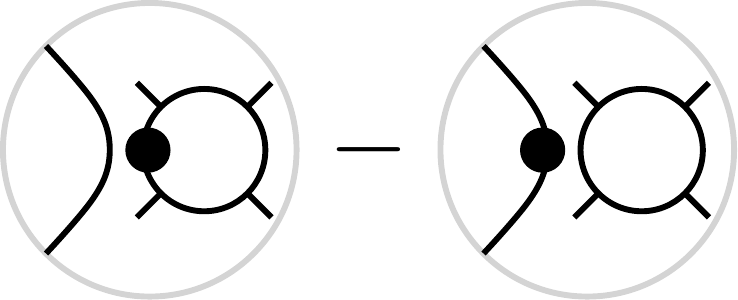}}
\newcommand{\kromc}{\HSimage{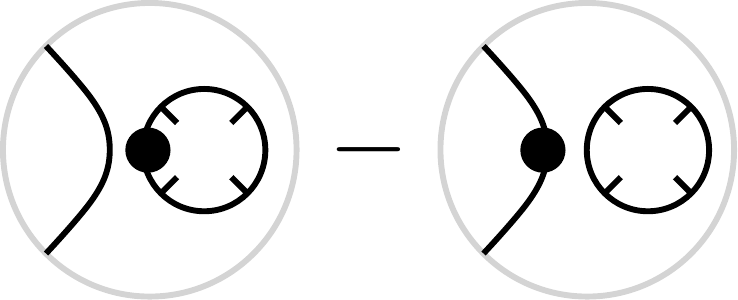}}
\newcommand{\kromd}{\HSimage{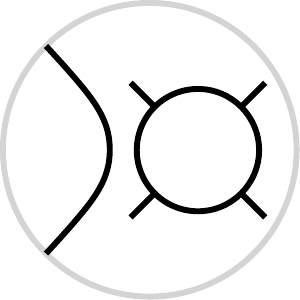}}

\newcommand{\krtcr}{\HSimage{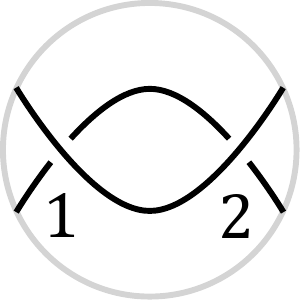}}
\newcommand{\krtcrl}{\HSimage{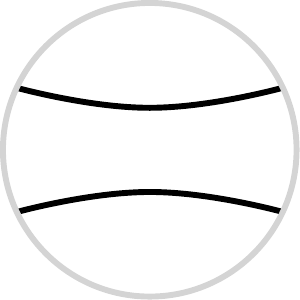}}
\newcommand{\krtsa}{\HSimage{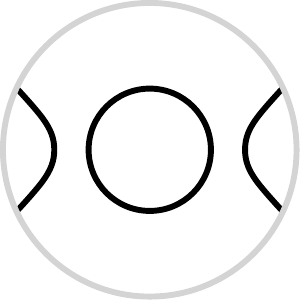}}
\newcommand{\krtsb}{\HSimage{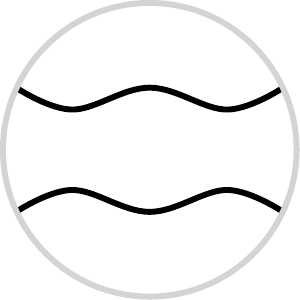}}
\newcommand{\krtsc}{\HSimage{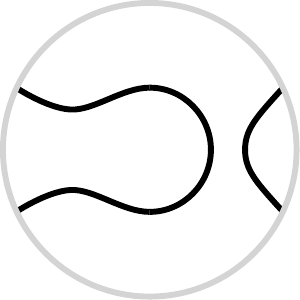}}
\newcommand{\krtsd}{\HSimage{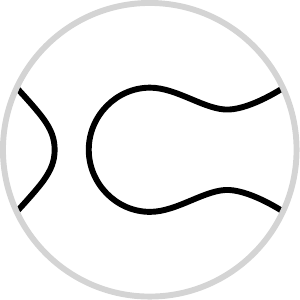}}
\newcommand{\krtma}{\HSimage{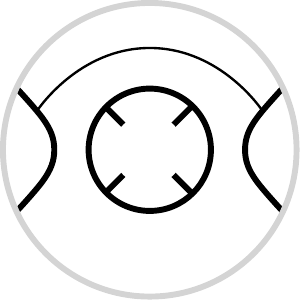}}
\newcommand{\krtmb}{\HSimage{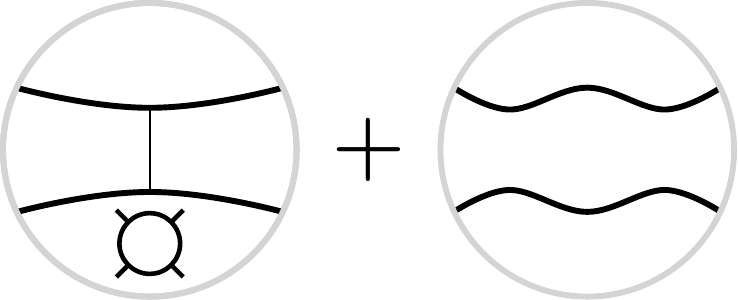}}

\begin{document}

\title{A new computational tool for Khovanov cobordism maps}
\author[Zs. Fehér]{Zsombor Fehér}
\address{Mathematical Institute, University of Oxford, Andrew Wiles Building, Radcliffe Observatory Quarter, Woodstock Road, Oxford, OX2 6GG, UK}
\email{zsombor012@gmail.com}
\begin{abstract}
We describe a Python module that we developed to calculate cobordism maps induced on Khovanov homology. As applications of our program, we compute these maps for all incompressible Seifert surfaces for prime knots up to 10 crossings, and distinguish many ribbon disks arising from symmetries of the boundary knots.
\end{abstract}
\maketitle

	\section*{Introduction}
Khovanov homology assigns an algebraic object $\KH(L)$ to every oriented link $L\sub S^3$. Computer programs \cites{knotjob,sage} are available to calculate this invariant, which can be used to distinguish links.

To any smoothly embedded surface $\S\sub S^3\times [0,1]$ between two links $L_0$ and $L_1$, Khovanov homology assigns a map $\KH(\S)\colon\KH(L_0)\to\KH(L_1)$. These maps provide a method to distinguish surfaces up to smooth isotopy. However, despite the combinatorial nature of their definition, no computational tool was available for calculating these maps. We developed a program to compute them, as detailed in \cref{sec:program}. This Python module is available on GitHub \cite{github}.

In \cref{sec:applications}, we present two applications. We calculate the cobordism maps for all incompressible Seifert surfaces for prime knots up to 10 crossings (\cref{table-Seifert}), and distinguish many ribbon disks arising from symmetries of the boundary knots (\cref{table-ribbon}). For instance, we present five distinct ribbon disks for the knot $10_{123}$, obtained via a rotational symmetry, which are distinguished by Khovanov homology (\cref{fig:ribbon}).

	\section{Background}
We give a brief overview of the definition of Khovanov homology, following \cite{Hayden-Sundberg}. Our focus is on the combinatorial details that are crucial for implementing the computer program. For a broader discussion of the context and philosophy underlying Khovanov homology, see \cites{Bar-Natan-tangles,Khovanov}. Throughout, we work with $\Z$ coefficients.

		\subsection{Khovanov homology of links}
Let $L\sub S^3$ be an oriented link. Choose a diagram $D$ of $L$, and an enumeration $X = (x_1,\dots,x_n)$ of its crossings. We define a chain complex $(\CKH(D, X), d)$ associated to the data $(D, X)$ as follows.

For any binary sequence $\sigma = (\sigma_1,\dots,\sigma_n)\in\{0,1\}^n$, we consider the \emph{smoothing}~$D_\sigma$, which is obtained by replacing each crossing $x_i$~{\crossing} with a 0-smoothing~{\zerosmoothing} or 1-smoothing~{\onesmoothing}, depending on $\sigma_i$. We call the connected components of $D_\sigma$ \emph{loops}. A \emph{labelled smoothing} $\alpha_\sigma$ is a labelling of the loops of $D_\sigma$ with a $1$ or $x$. Let $\CKH(D, X)$ be the free $\Z$-module generated by all the labelled smoothings $\alpha_\sigma$ of $D$.

For simplicity, we describe the differential map $d$ using the Morse induced chain maps that are defined in \cref{table_morse} (taken from \cite{Hayden-Sundberg}). Let $\alpha_\sigma$ be a labelled smoothing with $\sigma = (\sigma_1,\dots,\sigma_n)$. For an $i$ with $\sigma_i=0$, let $m_i$ denote the chain map induced by the Morse saddle move {\saddle} that turns the 0-smoothing of $x_i$ into the 1-smoothing. (Thus, $m_i(\alpha_\sigma)$ is a sum of 0, 1, or 2 labelled smoothings.) We define the differential map $d\colon \CKH(D, X)\to\CKH(D, X)$ on the generators by
\[
d(\alpha_\sigma) = \sum_{\substack{i \\ \sigma_i=0 }} (-1)^{\xi_i} m_i(\alpha_\sigma),
\qquad
\xi_i = \sum_{\substack{j \\ j<i}} \sigma_j,
\]
and extend it linearly to $\CKH(D, X)$. Then it can be verified that $d^2 = 0$, showing that $(\CKH(D, X), d)$ is indeed a chain complex. Its homology is denoted by $\KH(D, X)$. It turns out that up to isomorphism, $\KH(D, X)$ is unaffected by our particular choices of $D$ and $X$ \cite{Khovanov}, and thus, it can be denoted by $\KH(L)$, the \emph{Khovanov homology} of the link $L$.

\begin{table}[t]
\begin{center}
\renewcommand{\arraystretch}{2.1}
\begin{tabular}{|c|c|l|}
\hline
\text{Morse move} & \text{Ornament} &  \multicolumn{1}{c|}{\text{Definition of chain map}} \\
\hline
birth & \birth & $\begin{array}{l} \hspace{0.95em} 1 \hspace{0.95em} \mapsto \kocirc \vspace{.425em} \end{array}$ \\
\hline
death & \death & $\begin{array}{l} \kocirc \mapsto \hspace{0.95em} 0 \\ \kxcirc \mapsto \hspace{0.95em} 1 \vspace{.425em} \end{array}$ \\
\hline
\multirow{2}{80pt}{\hspace{2.3em}\vspace{-4.4em}saddle} & $\begin{array}{l} \vspace{-6em}\saddle \end{array}$ & $\begin{array}{l} \kcobdoo \mapsto \kcobio \\ \kcobdox \mapsto \kcobix \\ \kcobdxo \mapsto \kcobix \\ \kcobdxx \mapsto \hspace{0.95em} 0 \vspace{.425em} \end{array}$ \\ \cline{3-3} & & $\begin{array}{l} \kcobdo \mapsto \kcobiox + \kcobixo \\ \kcobdx \mapsto \kcobixx \vspace{.425em} \end{array}$ \\
\hline
\end{tabular}
\renewcommand{\arraystretch}{1}
\end{center}
\caption{Chain maps induced by a Morse move.}
\label{table_morse}
\end{table}

The chain complex $\CKH(D, X)$ is bigraded by homological grading $h$ and quantum grading $q$. Let $n_+,n_-$ be the number of positive and negative crossings in $D$. For a labelled smoothing $\alpha_\sigma$, let $|\sigma|$ be the number of 1-smoothings in $\sigma$, and let $v_+(\alpha_\sigma), v_-(\alpha_\sigma)$ be the number of 1-labels and $x$-labels in $\alpha_\sigma$. The grading $(h,q)$ of $\alpha_\sigma$ is defined by
\begin{align*}
h(\alpha_\sigma) &= |\sigma| - n_-, \\
q(\alpha_\sigma) &= v_+(\alpha_\sigma) - v_-(\alpha_\sigma) + h(\alpha_\sigma) + n_+ - n_-,
\end{align*}
and $\CKH^{h,q}(D, X)$ is defined as the submodule generated by such $\alpha_\sigma$. Then the differential map is also bigraded, $d\colon \CKH^{h,q}(D, X) \to \CKH^{h+1,q}(D, X)$. Hence, the bigrading descends to homology, and we have that $\KH(L) = \bigoplus_{h,q}\KH^{h,q}(L)$ is a bigraded $\Z$-module.

		\subsection{Crossing order change} \label{sec:crossing-order-change}
We explicitly construct the isomorphism of $\CKH(D, X)$ for a different choice of $X$, as this will be essential for computations. Consider two enumerations $X = (x_1,\dots,x_n)$ and $X' = (x_{\pi(1)},\dots,x_{\pi(n)})$ for some permutation $\pi$. Then any labelled smoothing $\alpha_\sigma\in \CKH(D, X)$ with $\sigma = (\sigma_1,\dots,\sigma_n)$ naturally corresponds to a labelled smoothing $\alpha'_{\sigma'}\in \CKH(D, X')$ that has $\sigma' = (\sigma_{\pi(1)},\dots,\sigma_{\pi(n)})$ and the same labels. By definition, the differential $d'$ on $\CKH(D, X')$ is given by
\[
d'(\alpha'_{\sigma'}) = \sum_{\substack{i \\ \sigma_i=0 }} (-1)^{\xi'_i} m_i(\alpha'_{\sigma'}),
\qquad
\xi'_i = \;\;\sum_{\mathclap{\substack{j \\ \pi^{-1}(j) < \pi^{-1}(i)}}} \;\;\sigma_j.
\]
We define the isomorphism $\Phi\colon \CKH(D, X) \to \CKH(D, X')$ on the generators by
\[
\Phi(\alpha_\sigma) = (-1)^{\pi_\sigma} \alpha'_{\sigma'},
\qquad
\pi_\sigma = \;\;\sum_{\mathclap{\substack{j, k \\ j<k \\ \pi^{-1}(j) > \pi^{-1}(k)}}} \;\;\sigma_j \sigma_k,
\]
and extend it linearly. Then it can be checked that $\Phi(d(\alpha_\sigma)) = d'(\Phi(\alpha_\sigma))$. Therefore, $(\CKH(D, X), d)$ and $(\CKH(D, X'), d')$ are isomorphic chain complexes, and $\KH(D, X)\cong \KH(D, X')$.

We also provide a more intuitive perspective on this isomorphism, by describing the differential map in a way that is independent of the crossing order, as follows. For a labelled smoothing $\alpha_\sigma$, let $i_1 < \dots < i_k$ represent all indices $i$ where $\sigma_i = 1$, and replace $\alpha_\sigma$ with the formal expression $\alpha_\sigma\cdot \omega_\sigma$, where $\omega_\sigma = dx_{i_1}\wedge\dots\wedge dx_{i_k}$. By substituting each labelled smoothing ($\alpha_\sigma$ and every $m_i(\alpha_\sigma)$) in this way within the definition of $d$, the differential map can be written formally as
\[
d = m_1\cdot dx_1+\dots+m_n\cdot dx_n,
\]
where $m_i\cdot dx_i$ acts on a labelled smoothing as
\[
(m_i\cdot dx_i)(\alpha_\sigma\cdot \omega_\sigma) = m_i(\alpha_\sigma)\cdot (dx_i\wedge\omega_\sigma).
\]
Indeed, the summation restricts to the indices $i$ with $\sigma_i = 0$ because $dx_i\wedge dx_i = 0$, and the signs $(-1)^{\xi_i}$ in our original definition of $d$ come from $dx_i\wedge dx_j = -dx_j\wedge dx_i$, when we change $dx_i\wedge \omega_\sigma$ to have increasing indices (as $m_i(\alpha_\sigma)$ has a new 1-smoothed crossing $x_i$).

With this order-independent view of $d$, it is clear what happens when we change the order of crossings to $X'$. The new $\omega'_{\sigma'}$ that uses the ordering $X'$ has $\omega'_{\sigma'} = \pm \omega_\sigma$, where the $\pm$ sign depends on the parity of the permutation of $dx_{i_1},\dots,dx_{i_k}$ in $\omega'_{\sigma'}$. In the above definition of $\Phi$, $\alpha_\sigma$ gets this sign, whose explicit definition is $(-1)^{\pi_\sigma}$. For example, if $X = (x_1, x_2, x_3)$ and $X' = (x_1, x_3, x_2)$, then
\[
\Phi(\alpha_{110} + \alpha_{101} + \alpha_{011})
=    \alpha_{101}'+ \alpha_{110}'- \alpha_{011}'.
\]

		\subsection{Khovanov homology of cobordisms}

Let $L_0,L_1$ be links. A \emph{cobordism} from $L_0$ to $L_1$ is a smoothly embedded compact surface $\S\sub S^3\times[0,1]$ whose boundary is $L_0\times\{0\}\cup L_1\times\{1\}$. Any cobordism can be represented as a \emph{movie}: a sequence of diagrams $D_0,\dots,D_N$, where $D_{i+1}$ is obtained from $D_i$ by an \emph{elementary move} (Reidemeister or Morse move), and $D_0,D_N$ represent the links $L_0,L_1$, respectively. We often write the cobordism as a function $\S\colon L_0\to L_1$.

Given a cobordism $\S\colon L_0\to L_1$, we will associate to $\S$ a group homomorphism $\KH(\S)\colon \KH(L_0)\to \KH(L_1)$. Moreover, if $\S$ is orientable, with oriented boundary $\d\S=(L_0\times\{0\})\cup (L_1^r\times\{1\})$ (where $L_1^r$ denotes $L_1$ with reversed orientation), then $\KH(\S)$ will be a bigraded map:
\[
\KH(\S)\colon \KH^{h,q}(L_0)\to \KH^{h,q+\chi(\S)}(L_1).
\]

We now give the definition of this map, using the movie $D_0,\dots,D_N$. We do this by defining a chain map $\CKH(D_i, X_i)\to \CKH(D_{i+1}, X_{i+1})$ associated to every Reidemeister and Morse move, and $\CKH(\S)$ will be the composition of these maps.

\subsubsection*{Morse moves}

There are three types of Morse moves: \emph{births}, \emph{deaths}, and \emph{saddles}. A birth~{\birth} adds a small new loop to the diagram, a death {\death} removes a small loop, and a saddle {\saddle} merges or splits loops, by changing {\zerosmoothing} to {\onesmoothing} locally.

Let $X_i$ be an enumeration of the crossings of $D_i$. If $D_{i+1}$ is obtained from $D_i$ by a Morse move, then the list of crossings does not change. We let $X_{i+1} = X_i$, and we define the chain map $\CKH(D_i, X_i)\to \CKH(D_{i+1}, X_{i+1})$ on labelled smoothings according to \cref{table_morse}.

\subsubsection*{Reidemeister moves}

There are three types of Reidemeister moves: R1 removes or adds 1 crossing, R2 removes or adds 2 crossings, and R3 manipulates the diagram around 3 crossings. We define the Reidemeister induced chain maps using the Morse induced chain maps defined previously. In addition, we will use {\kdot} to denote two saddle moves on an arc (one splitting, then one re-merging). Then $\frac12$~{\kdot} sends a 1-labelled loop to an $x$-labelled loop, and an $x$-labelled loop to 0.

Let $X_i$ be an enumeration of the crossings of $D_i$. If $D_{i+1}$ is obtained from $D_i$ by an R1 move that adds a crossing $x$, then add $x$ to the enumeration in the last position: $X_{i+1} = (X_i, x)$. We define the chain map $\CKH(D_i, X_i)\to \CKH(D_{i+1}, X_{i+1})$ on labelled smoothings according to \cref{table_r1}. If the R1 move removes a crossing $x$, then consider first an enumeration $X_i'$ that has $x$ as its last element, and apply the isomorphism $\CKH(D_i, X_i)\to \CKH(D_i, X_i')$ described in \cref{sec:crossing-order-change}. Then we define $\CKH(D_i, X_i')\to \CKH(D_{i+1}, X_{i+1})$ according to the same Table, where $X_i' = (X_{i+1}, x)$.

If $D_{i+1}$ is obtained from $D_i$ by an R2 move that adds the crossings $x_1,x_2$, then let $X_{i+1} = (X_i, x_1, x_2)$, and define the chain map $\CKH(D_i, X_i)\to \CKH(D_{i+1}, X_{i+1})$ on labelled smoothings according to \cref{table_r2}. If the R2 move removes two crossings $x_1,x_2$, then first apply an isomorphism $\CKH(D_i, X_i)\to \CKH(D_i, X_i')$ that moves $x_1,x_2$ to the last two positions in $X_i'$, then define $\CKH(D_i, X_i')\to \CKH(D_{i+1}, X_{i+1})$ according to the same Table, where $X_i' = (X_{i+1}, x_1, x_2)$.

\begin{table}
\begin{center}
\renewcommand{\arraystretch}{2.1}
\begin{tabular}{|c|c|c|}
\hline
Reidemeister move & Smoothing & Induced map \\
\hline
\multirow{2}{68pt}{\begin{minipage}{9em} \vspace{8pt} $\kropos \to \kroarc$ \end{minipage}} & 
\krosa & \kroma $\begin{array}{c}\hspace{-1em}\vspace{.425em}\end{array}$ \\
\cline{2-3} & \krosb & $\begin{array}{c} 0 \vspace{.425em} \end{array}$ \\
\hline
$\kroarc \to \kropos$ & \kroarc & $\begin{array}{c} \frac12 \bigg( \kromb \bigg) \vspace{.425em} \end{array}$ \\
\hline
\multirow{2}{68pt}{\begin{minipage}{9em} \vspace{10pt} $\kroneg \to \kroarc$ \end{minipage}} & 
\krosa & $\begin{array}{c} \frac12 \bigg( \kromc \bigg) \vspace{.425em} \end{array}$ \\
\cline{2-3} & \krosb & $\begin{array}{c} 0 \vspace{.425em} \end{array}$ \\
\hline
$\kroarc \to \kroneg$ & \kroarc & \kromd $\begin{array}{c}\hspace{-1em}\vspace{.425em}\end{array}$ \\
\hline
\end{tabular}
\renewcommand{\arraystretch}{1}
\end{center}
\caption{Chain maps induced by a Reidemeister 1 move.}
\label{table_r1}
\end{table}

\begin{table}
\begin{center}
\renewcommand{\arraystretch}{2.1}
\begin{tabular}{|c|c|c|}
\hline
Reidemeister move & Smoothing & Induced map \\
\hline
\multirow{2}{68pt}{\begin{minipage}{9em} \vspace{47pt} $\krtcr \to \krtcrl$ \end{minipage}} & 
\krtsa & $-$ \!\! \krtma $\begin{array}{c}\hspace{.25em}\vspace{.425em}\end{array}$ \\
\cline{2-3} & \krtsb & \krtcrl $\begin{array}{c}\hspace{-1em}\vspace{.425em}\end{array}$ \\
\cline{2-3} & \krtsc & $\begin{array}{c} 0 \vspace{.425em} \end{array}$ \\
\cline{2-3} & \krtsd & $\begin{array}{c} 0 \vspace{.425em} \end{array}$ \\
\hline
$\krtcrl \to \krtcr$ & \krtcrl & \krtmb $\begin{array}{c}\hspace{-1em}\vspace{.425em}\end{array}$ \\
\hline
\end{tabular}
\renewcommand{\arraystretch}{1}
\end{center}
\caption{Chain maps induced by a Reidemeister 2 move.}
\label{table_r2}
\end{table}

\begin{table}
\begin{center}
\begin{tabular}{c}
\includegraphics[scale=.57]{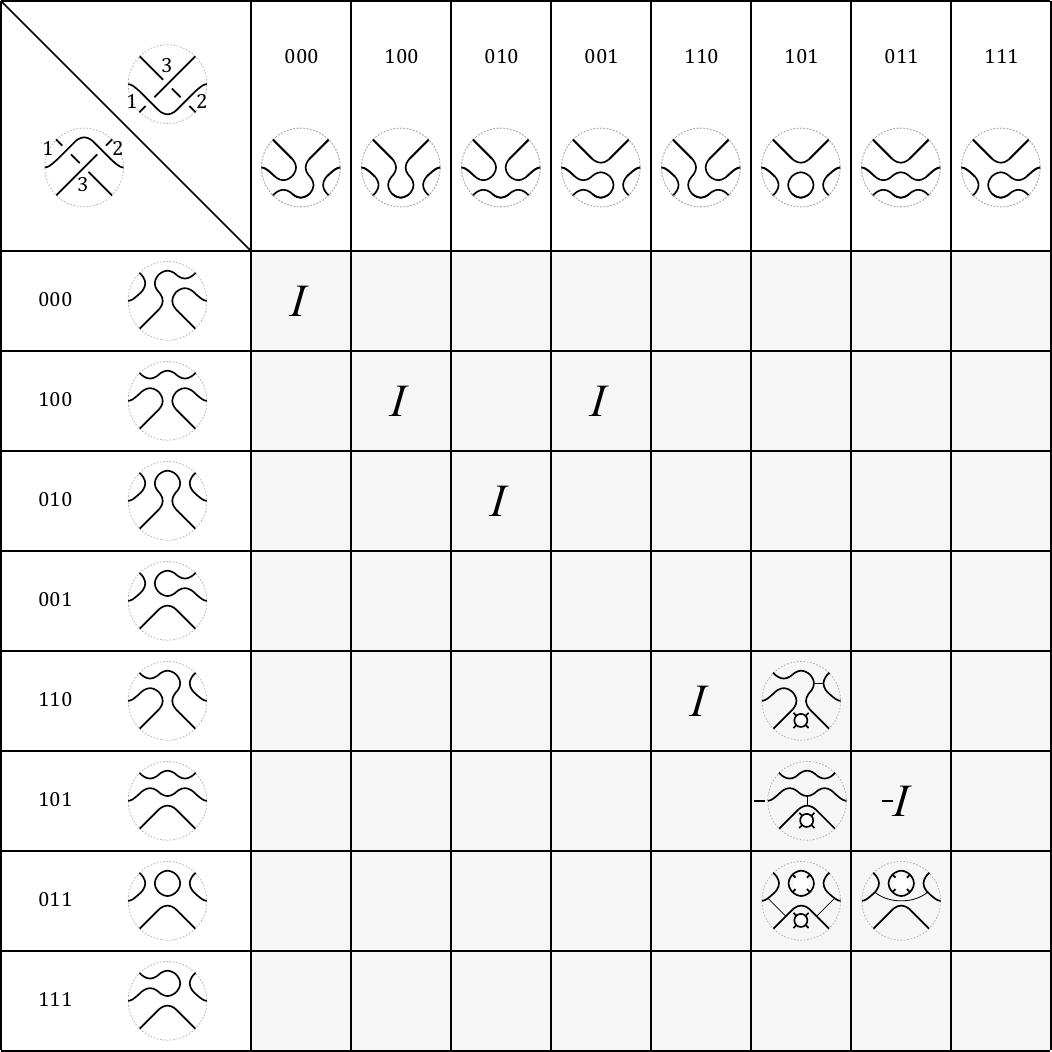} \\
\includegraphics[scale=.57]{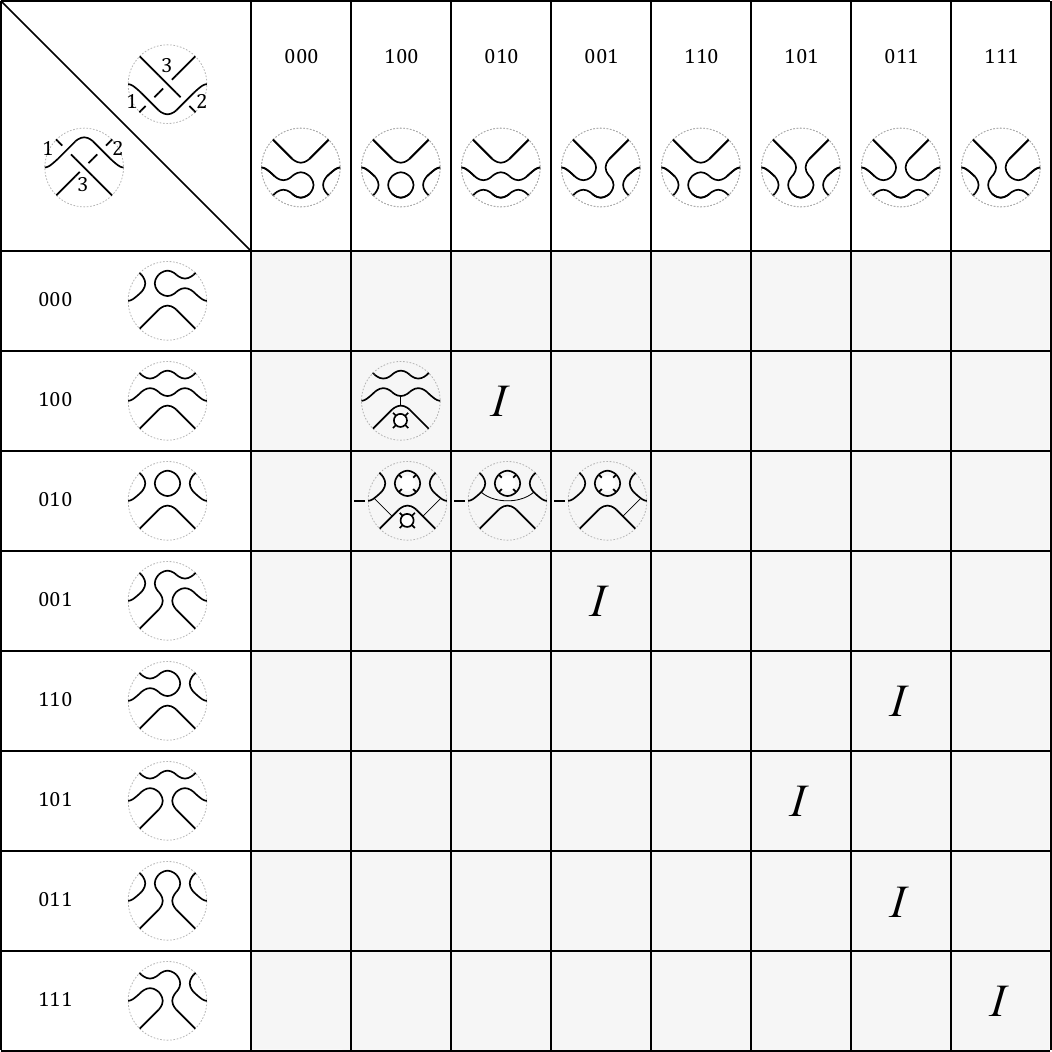}
\end{tabular}
\end{center}
\caption{Chain maps induced by a Reidemeister 3 move.}
\label{table_r3}
\end{table}

Finally, assume that $D_{i+1}$ is obtained from $D_i$ by an R3 move. The chain maps are presented in table form (\cref{table_r3}), interpreted as follows. (These tables are adapted from \cite{Hayden-Sundberg}, but with corrected signs.) The top-left cell of each table depicts the Reidemeister move, showing the starting diagram with crossings $x_1,x_2,x_3$ in the left corner, and the ending diagram with crossings $y_1,y_2,y_3$ in the right corner. As usual, we first apply an isomorphism $\CKH(D_i, X_i)\to \CKH(D_i, X_i')$, which reorders the crossings $x_1,x_2,x_3$ to the last three positions. Write $X_i'=(X,x_1,x_2,x_3)$, and let $X_{i+1}=(X,y_1,y_2,y_3)$. The map $\CKH(D_i, X_i')\to \CKH(D_{i+1}, X_{i+1})$ is then defined as follows. The left column of the table lists the smoothings of the starting diagram, while the top row lists the smoothings of the ending diagram. The image of a labelled smoothing is obtained by adding together the non-empty cells in its corresponding row, where $I$ denotes an isotopy of the labelled smoothing.

This completes the definition of the maps induced by elementary moves. Now, let $\S\colon L_0\to L_1$ be a cobordism with movie $D_0,\dots,D_N$, and let $X_0$ and $X_N'$ be enumerations of the crossings of $D_0$ and $D_N$, respectively. We obtain enumerations $X_1,\dots,X_N$ and maps $\phi_i\colon\CKH(D_i,X_i)\to\CKH(D_{i+1},X_{i+1})$ from the description above. We define
\[
\CKH(\S,X_0,X_N') = \Phi\circ\phi_{N-1}\circ\dots\circ\phi_0\colon \CKH(D_0,X_0)\to \CKH(D_N,X_N'),
\]
where $\Phi\colon\CKH(D_N,X_N)\to \CKH(D_N,X_N')$ is the isomorphism induced by reordering the crossings. It can be checked that any $\phi_i$ is a chain map. Therefore, $\CKH(\S,X_0,X_N')$ is also a chain map, which induces a map on homology: $\KH(\S)\colon \KH(L_0)\to \KH(L_1)$.

The power of using Khovanov homology for cobordisms comes from the fact that the map $\KH(\S)$ is invariant (up to a global sign) under smooth isotopy of $\S$ relative to the boundary \cite{Jacobsson}. In fact, the same holds if we only look at a diffeomorphism of the ambient space:

\begin{theorem}[\cite{Lipshitz-Sarkar}*{Proposition 3.7}]\label{thm:Jacobsson}
Let $\S\sub S^3\times[0,1]$ be a cobordism from $L_0$ to $L_1$. Then the map $\KH(\S)\colon \KH(L_0)\to\KH(L_1)$ is invariant up to multiplication by $\pm1$ under any diffeomorphism of $S^3\times[0,1]$ that restricts to the identity on the boundary.
\end{theorem}

As often is the case, if $\S\sub D^4$ is a slice surface bounded by the link $L\sub S^3$, then $\S$ can be viewed as a cobordism $\S\colon L\to \emptyset$, or as $\S^\star\colon \emptyset \to L$. Thus, to distinguish the surfaces $\S_0,\S_1\sub D^4$ smoothly, it suffices to distinguish the maps $\KH(\S_i)$, or the maps $\KH(\S_i^\star)$. Recent examples in the literature have demonstrated the use of this method to distinguish surfaces smoothly or exotically \cites{Sundberg-Swann,Hayden-Sundberg,HKMPS,Hayden-atomic}.

If $\S\colon \emptyset\to L$ is a cobordism, then the map $\KH(\S)\colon \KH(\emptyset)\to\KH(L)$ is entirely determined by the image of $1\in\Z=\KH(\emptyset)$. We call this homology class $\KH(\S)(1)$ the \emph{Khovanov--Jacobsson class} of $\S$.

To determine whether two Khovanov--Jacobsson classes are the same up to sign, part~(b) of the following well-known statement often provides a simple diagrammatic proof that a chain element is non-zero in homology.

\begin{prop}\label{prop:outside}
Let $\alpha\in\CKH(D, X)$ be a chain element, given as a linear combination $\alpha=n_1\alpha_1+\dots+n_s\alpha_s$ of distinct labelled smoothings $\alpha_i$, $0\ne n_i\in\Z$.
\begin{myenumerate}
\item If, for every $\alpha_i$, each 0-smoothed crossing in $\alpha_i$ connects two distinct $x$-labelled loops, then $d\alpha=0$, i.e. $\alpha$ is a cycle.
\item If there exists an $\alpha_i$ where each 1-smoothed crossing in $\alpha_i$ connects two distinct 1-labelled loops, then $\alpha\ne d\beta$ for any $\beta$, i.e. $\alpha$ is not a boundary.
\end{myenumerate}
\end{prop}

\begin{proof}
This is straightforward from the definition of $d$ and the Morse saddle induced chain maps (\cref{table_morse}).
\end{proof}

		\subsection{Mirroring}\label{sec:mirroring}

Let $-L$ denote the mirror of the link $L$, and $-D$ the mirror of the diagram $D$, obtained by swapping overstrands and understrands at every crossing. Then there is an isomorphism of chain complexes $\CKH(-D,X)\cong \CKH(D,X)^\star$, where $C^\star = \Hom(C,\Z)$ is the dual of the $\Z$-module $C$ \cite{Khovanov}*{Proposition 32}. This isomorphism $\Psi_D$ is given by $\alpha_\sigma\mapsto (\beta_{-\sigma})^\star$, where $(-\sigma)_i = 1-\sigma_i$ is the same smoothing on the mirrored diagram, $\beta_{-\sigma}$ is obtained from $\alpha_\sigma$ by swapping 1-labels and $x$-labels, and for a $\Z$-module $C$ freely generated by $c, c_1, c_2,\ldots\in C$, the homomorphism $c\mapsto 1, c_i\mapsto 0$ is denoted by $c^\star\in C^\star$.

If $f\colon C\to B$ is a map, then the dual $f^\star\colon B^\star\to C^\star$ is defined by $\phi\mapsto \phi\circ f$ for $\phi\in B^\star$.
Then it is straightforward to check that $\Psi_D\circ d_{-D}=d_D^\star\circ\Psi_D$, showing the isomorphism of chain complexes $(\CKH(-D,X), d_{-D})$ and $(\CKH(D,X)^\star, d_D^\star)$. Taking homology, this shows that $\KH(-L)$ is related to $\KH(L)$ in the usual way that homology is related to cohomology. If $\Tor$ denotes the torsion subgroup, this means that, since $\Psi$ multiplies both gradings by $-1$,
\begin{align*}
\KH^{h,q}(-L) \otimes \Q &\cong \KH^{-h,-q}(L) \otimes \Q, \\
\Tor(\KH^{h,q}(-L))      &\cong \Tor(\KH^{1-h,-q}(L)).
\end{align*}

Let $\S\colon L_0\to L_1$ be a cobordism. The reverse of the cobordism is $\S^\star\colon L_1\to L_0$, and the mirror of the cobordism is $-\S\colon {-L_0}\to {-L_1}$. Then the cobordism maps induced on Khovanov homology are compatible with the previous isomorphism. Specifically, the following diagram commutes up to chain homotopy\footnote{It appears that a proof of this statement does not appear in the literature. It is stated in \cites{Hayden-lecture, Sundberg-PhD}, but without a proper reference.}:
\[
\begin{tikzcd}
\CKH(-D_0, X_0) \arrow[rr, "\Psi_{D_0}"', "\cong"] \arrow[d, "\CKH(-\S)"'] &  & \CKH(D_0, X_0)^\star \arrow[d, "\CKH(\S^\star)^\star"] \\
\CKH(-D_1, X_1) \arrow[rr, "\Psi_{D_1}"', "\cong"]                        &  & \CKH(D_1, X_1)^\star                                 
\end{tikzcd}
\]

	\section{The computer program}\label{sec:program}

We have developed a Python module that can calculate the map $\KH(\S)$ for any cobordism~$\S$ (given as a movie). This addresses an important gap in the existing computational tools for Khovanov homology. The source code is available in the file \mintinline{text}{khovanov.py} on GitHub \cite{github}, and is also included as an ancillary file in the arXiv version of this paper.

As a basic example, we use the following code to show that the two annuli $A_0,A_1$, into which the torus link $T_{4,6}$ divides the torus, are not smoothly isotopic relative to the boundary in $D^4$ \cite{Feher-pqkn}*{Theorem 1}.

\medskip
\begin{minted}[linenos]{python}
from khovanov import *

L = Link(braid_closure = [1, 2, 3] * 6)
c, d = L.crossings[0], L.crossings[3]
L.orient((c, 0), (c, 1))

A0 = Cobordism(L)
A0.morse_saddle((c, 0), (c, 3))
A0.finish()

A1 = Cobordism(L)
A1.morse_saddle((d, 0), (d, 3))
A1.finish()

compare(A0.KJ_class(), A1.KJ_class())
\end{minted}
\medskip

Lines 3--5 define the link $L=T_{4,6}$ using a braid closure, denote two of its crossings $c$ and $d$, and define its orientation (needed only when the link has more than one component). Lines 7--9 define the annulus $A_0$ combinatorially, viewed as a cobordism from $L$ to $\emptyset$. Its movie starts with a Morse saddle move between two strands of the diagram that are identified by the 0th and 3rd strands around the crossing $c$. The next moves of the movie would then be added as subsequent function calls (e.g. \code{A0.morse_birth()}, \code{A0.reidemeister_2((d, 3), (d, 0))}, \code{A0.reidemeister_1_up((c, 2))}). In this example, however, the resulting link after the first move is the unknot, so we use the convenient \code{finish()} function to automatically finish the rest of the cobordism. Lines 11--13 define the cobordism $A_1$ similarly, starting with a different Morse saddle move. In line~15, we calculate chain representatives of the Khovanov--Jacobsson classes of $A_0^\star, A_1^\star$, and compare them in homology. In this case, the program outputs that both their sum and difference are trivially outside the image of the differential map $d$ (\cref{prop:outside} (b)), giving a proof that $A_0, A_1$ are not smoothly isotopic in $D^4$ by \cref{thm:Jacobsson}.

		\subsection{Classes and methods}

We provide a brief overview of the program's classes, highlighting key properties and methods that might be useful to users. The program builds on SnapPy's \code{Link} class \cite{SnapPy}, with modifications. This class represents a link diagram combinatorially as a graph in which each vertex has degree 4 or 2.

		\subsection*{Crossing}
A \code{Crossing} represents a crossing in the link diagram, modelled as a degree-4 vertex. The strands are labelled 0, 1, 2, and 3 in positive orientation around the vertex, with 0 and 2 corresponding to the understrand. Its properties:
\begin{itemize}
\item \code{adjacent} specifies the connection for each strand.
\item \code{label} is used to identify a crossing within a \code{Link} (usually an integer).
\end{itemize}
For example, \code{a = Crossing(7)} creates a \code{Crossing} object with \code{label} 7, and connecting the strand \code{(a, 3)} to \code{(b, 0)} can be achieved by writing \code{a[3] = b[0]}.

		\subsection*{Strand}
A \code{Strand} is similar to a \code{Crossing}, but represents a degree-2 vertex. While SnapPy automatically fuses these and removes unknotted unlinked components, we retain them as they are essential for representing cobordism movies.

		\subsection*{Link}
A \code{Link} represents a link diagram with an enumeration of its crossings. Its property:
\begin{itemize}
\item \code{crossings} provides a list of \code{Crossing} and \code{Strand} objects that form the vertices of the graph.
\end{itemize}
Links can be created using PD codes, braid closures, or a list of \code{Crossing} and \code{Strand} objects, e.g. \code{L = Link([a, b, c])}. Its methods:
\begin{itemize}
\item \code{L.orient(cs0, cs1, ...)}: if the link $L$ has multiple components, it may be necessary to define its orientation to set its \code{n_plus} and \code{n_minus} properties (representing the number of positive and negative crossings).
\item \code{L.differential_matrix(h, q)} returns a matrix of the differential map $d$ in grading $(h,q)\to (h+1,q)$.
\item \code{L.homology_with_generators(h, q)} calculates the Khovanov homology $\KH^{h,q}(L)$, expressing it as a direct sum of cyclic groups and returning chain representatives for the generators of each group (this method needs Sage \cite{sage}).
\end{itemize}
For computing Khovanov homology without generators, significantly faster programs exist, such as KnotJob \cite{knotjob}, which implements Bar-Natan's efficient algorithm \cite{Bar-Natan-algorithm}.

		\subsection*{CrossingStrand}
A \code{CrossingStrand} represents one of the four strands around a \code{Crossing} (or one of the two strands around a \code{Strand}). Its properties:
\begin{itemize}
\item \code{crossing} references the associated \code{Crossing} or \code{Strand} object.
\item \code{strand_index} indicates the strand's position, e.g. 0, 1, 2, or 3.
\end{itemize}
For example, \code{cs = CrossingStrand(a, 2)} sets \code{cs.crossing} to the object \code{a}, and \code{cs.strand_index} to 2.

\code{CrossingStrand} objects are crucial for defining cobordisms, as they allow precise specification of strands involved in an elementary move. Several methods are available to navigate within a \code{Link}:
\begin{itemize}
\item \code{cs + 1}, \code{cs.next()}, and \code{cs - 1} give the adjacent strands around the same crossing.
\item \code{cs.opposite()} returns the strand located at the opposite end of the edge connected to \code{cs}.
\end{itemize}
It is important to understand that while \code{cs} and \code{cs.opposite()} lie on the same edge of the underlying graph of the link diagram, they represent distinct portions of the edge. Care must be taken to use the appropriate one when specifying strands for defining an elementary move.

		\subsection*{SmoothLink}
A \code{SmoothLink} represents a smoothing of a link diagram. Its properties:
\begin{itemize}
\item \code{link} refers to the underlying \code{Link}.
\item \code{smoothing} is a dictionary mapping each \code{Crossing} to 0 or 1.
\item \code{loops} is a list of loops, where each loop is represented as a tuple of \code{CrossingStrand} objects (two at each crossing) traversed along the loop.
\end{itemize}

		\subsection*{LabelledSmoothing}
A \code{LabelledSmoothing} represents a labelled smoothing with a coefficient. Its properties:
\begin{itemize}
\item \code{smooth_link} refers to the underlying \code{SmoothLink}.
\item \code{labels} is a dictionary mapping each loop to \code{"1"} or \code{"x"}.
\item \code{coefficient} is an integer that will represent the coefficient in a linear combination.
\end{itemize}
There are several methods for examining the relationship between a \code{LabelledSmoothing} \code{LS} and the image of the differential map (without computing the entire differential matrix):
\begin{itemize}
\item \code{LS.differential()} returns its differential as a list of \code{LabelledSmoothing} objects, interpreted as a linear combination.
\item \code{LS.is_outside()} checks if every 1-smoothed crossing connects two distinct 1-labelled loops. If so, then \code{LS} lies outside the image of the differential by \cref{prop:outside} (b).
\item \code{LS.row_size()} and \code{LS.reverse_differential()} give how many and which elements contain \code{LS} in their differential.
\end{itemize}

		\subsection*{CKhElement}
A \code{CKhElement} represents a chain element of $\CKH(D,X)$. It is a list of \code{LabelledSmoothing} objects, interpreted as a linear combination using their \code{coefficient} properties.

Each of the elementary moves can be applied to a \code{CKhElement} \code{CKH} by calling the corresponding method and specifying the location with appropriate \code{CrossingStrand} objects (see the source code for more details and pictures):

\begin{itemize}
\item \code{CKH.morse_birth()} adds a new loop.
\item \code{CKH.morse_death(cs)} removes a loop containing \code{cs}.
\item \code{CKH.morse_saddle(cs0, cs1)} adds a saddle move to connect \code{cs0} with \code{cs1}.
\item \code{CKH.reidemeister_1(cs)} removes a twist containing \code{cs}.
\item \code{CKH.reidemeister_1_up(cs, True)} adds a positive twist to the left of \code{cs}.
\item \code{CKH.reidemeister_2(cs0, cs1)} removes the common crossing of \code{cs0} and \code{cs1}, and the common crossing of \code{cs0.opposite()} and \code{cs1.opposite()}.
\item \code{CKH.reidemeister_2_up(cs0, cs1)} moves \code{cs0} to the right and over \code{cs1}.
\item \code{CKH.reidemeister_3(cs0, cs1, cs2)} moves \code{cs0} over \code{cs1}, \code{cs2}, and the common crossing of \code{cs1.opposite()} and \code{cs2.opposite()}.
\end{itemize}
These change the underlying \code{LabelledSmoothing}, \code{SmoothLink}, and \code{Link} objects as well by the elementary move. Further methods:
\begin{itemize}
\item \code{CKH.fuse(strand)} fuses a \code{Strand}, combining the two edges from it.
\item \code{CKH.differential()} returns the differential.
\item \code{CKH.reorder_crossings(new_order)} reorders the crossings (see \cref{sec:crossing-order-change}).
\item \code{CKH.replace_link(link, flipping)} might be useful if \code{CKH} is on a different copy of the same link.
\item \code{CKH.simplify()} simplifies the linear combination by combining equal terms and removing 0-coefficient elements.
\item \code{CKH0 + CKH1} and \code{CKH0 - CKH1} calculate the sum and difference of two chain elements and simplify the result.
\item \code{compare(CKH0, CKH1)} returns whether two chain elements are the same up to sign in homology.
\end{itemize}

		\subsection*{Cobordism}
A \code{Cobordism} represents a movie of a cobordism between two links. Its properties:
\begin{itemize}
\item \code{movie} contains a list of elementary moves and their locations.
\item \code{links} gives the list of links corresponding to each stage of the movie.
\end{itemize}
A new cobordism $C\colon L\to L$ with empty movie can be created by \code{C = Cobordism(L)}, and moves can be added by subsequent calls of methods. Crossing labels can also be used in arguments instead of \code{CrossingStrand} objects, e.g. \code{C.morse_death((7, 1))}. In addition to the above methods of \code{CKhElement}, there are two convenient methods for adding multiple elementary moves at once: \code{band_move(...)} and \code{finish()}.

\begin{itemize}
\item \code{C.band_move(n, cs0, (cs1, True), (cs2, False), ..., csk)} adds a band with $n$ half-twists, starting at \code{cs0}, passing over \code{cs1}, under \code{cs2}, and so on, before ending at \code{csk}.
\item \code{C.finish()} tries to simplify the last link using Reidemeister~1--3 moves (similar to SnapPy's \code{simplify("level")} algorithm), then does a Morse death on each crossingless component.
\item \code{C.reverse()} returns the reversal of the cobordism.
\item \code{C.mirror()} returns the mirror of the cobordism.
\item \code{C.chi()} returns the Euler characteristic of the cobordism.
\item \code{C.map(CKH)} applies the map induced by the cobordism to \code{CKH} and changes it in place.
\item \code{C.KJ_class()} returns the Khovanov--Jacobsson class of \code{C} or \code{C.reverse()}, if one of the ends of \code{C} is the empty link.
\item \code{C.matrix(h, q)} calculates the cobordism map induced on homology in grading $(h,q)\to (h,q+\chi(C))$ in matrix form (this method needs Sage).
\end{itemize}

\medskip
All classes support the standard \code{print(...)} method to display information about an object. For instance, \code{print(L)} outputs the adjacency structure of a \code{Link}, \code{print(C)} displays the sequence of moves in a \code{Cobordism} movie, and \code{print(CKH)} represents a \code{CKhElement} by showing the labels of each loop in every element of the linear combination. Additionally, \code{CKH.print_short()} provides a concise representation that omits printing the loops.

	\section{Applications}\label{sec:applications}

As two applications of our program, we compute cobordism maps for Seifert surfaces and ribbon disks bounded by small knots.

		\subsection{Seifert surfaces up to 10 crossings}

For a prime knot $L$ with at most 10 crossings, there is a unique incompressible Seifert surface $\S$ up to smooth isotopy in $D^4$ (see \cite{Kakimizu} for the classification up to isotopy in $S^3$). Viewing $\S$ as a cobordism $L\to\emptyset$, the induced map on Khovanov homology is $\KH(\S)\colon\KH(L)\to \KH(\emptyset)$. As $\KH(\emptyset)\cong\Z$ is supported in grading $(0, 0)$, and $\KH(\S)$ shifts grading by $(0,\chi)$, where $\chi=\chi(\S)$, the only non-zero part of this cobordism map is $\KH^{0,-\chi}(L)\to \KH^{0,0}(\emptyset)$. For all knots $L$ considered, $\KH^{0,-\chi}(L) \cong \Z^k$ is torsion-free. When $k\ge 1$, this map is represented by a $1\times k$ matrix. As the isomorphism with $\Z^k$ is not canonical, the only information contained in this matrix is the greatest common divisor of its entries. We now calculate this non-negative integer for all such knots $L$ and their mirror.

If $D$ is an alternating diagram of $L$, then Seifert's algorithm produces a minimal genus Seifert surface for $L$ \cite{Gabai}. This algorithm constructs the surface from the orientation-preserving smoothing $\sigma$ by attaching a half-twisted band at each crossing to $D_\sigma$. This method provides a straightforward way to algorithmically generate incompressible Seifert surfaces for most prime knots with at most 10 crossings.

{
\setlength{\tabcolsep}{4pt}
\begin{table}
\begin{center}
\begin{small}
\begin{tabular}{lll|lll|lll|lll|lll}
$L$ & $\S$ & \hspace{-0.7em} $-\S$ \hspace{-1em}
\\ \hline
$3_1$ & 1 &  &        $9_{16}$ & 1 &  &         $10_{17}$ &  &  &           $10_{67}$ & 2\;0 &  &     $10_{117}$ &  &  \\
$4_1$ & 2 & 2 &       $9_{17}$ &  &  &          $10_{18}$ & 2\;0\;0\;0 &  & $10_{68}$ &  &  &         $10_{118}$ &  &  \\
$5_1$ & 1 &  &        $9_{18}$ & 1 &  &         $10_{19}$ &  &  &           $10_{69}$ &  &  &         $10_{119}$ &  &  \\
$5_2$ & 1 &  &        $9_{19}$ &  &  &          $10_{20}$ & 2 &  &          $10_{70}$ &  &  &         $10_{120}$ & 1 &  \\
$6_1$ & 2 & 2\;0 &    $9_{20}$ & 2\;0 &  &      $10_{21}$ & 2\;0 &  &       $10_{71}$ &  &  &         $10_{121}$ &  &  \\
$6_2$ & 2 &  &        $9_{21}$ &  & 2\;0\;0 &   $10_{22}$ &  &  &           $10_{72}$ & 2\;0 &  &     $10_{122}$ &  &  \\
$6_3$ &  &  &         $9_{22}$ &  &  &          $10_{23}$ &  &  &           $10_{73}$ &  &  &         $10_{123}$ &  &  \\
$7_1$ & 1 &  &        $9_{23}$ & 1 &  &         $10_{24}$ & 2\;0 &  &       $10_{74}$ & 2\;0\;0 &  &  $10_{124}$ & 1 &  \\
$7_2$ & 1 &  &        $9_{24}$ &  &  &          $10_{25}$ & 2\;0 &  &       $10_{75}$ &  &  &         $10_{125}$ &  &  \\
$7_3$ & 1 &  &        $9_{25}$ & 2\;0 &  &      $10_{26}$ &  &  &           $10_{76}$ & 2 &  &        $10_{126}$ &  &  \\
$7_4$ & 1 &  &        $9_{26}$ &  &  &          $10_{27}$ &  &  &           $10_{77}$ &  &  &         $10_{127}$ & 2 &  \\
$7_5$ & 1 &  &        $9_{27}$ &  &  &          $10_{28}$ &  &  &           $10_{78}$ & 2\;0\;0 &  &  $10_{128}$ & 1 &  \\
$7_6$ & 2\;0 &  &     $9_{28}$ &  &  &          $10_{29}$ &  &  &           $10_{79}$ &  &  &         $10_{129}$ &  &  \\
$7_7$ &  &  &         $9_{29}$ &  &  &          $10_{30}$ & 2\;0\;0 &  &    $10_{80}$ & 1 &  &        $10_{130}$ &  &  \\
$8_1$ & 2 & 2\;0 &    $9_{30}$ &  &  &          $10_{31}$ &  &  &           $10_{81}$ &  &  &         $10_{131}$ & 2 &  \\
$8_2$ & 2 &  &        $9_{31}$ &  &  &          $10_{32}$ &  &  &           $10_{82}$ &  &  &         $10_{132}$ &  & 2 \\
$8_3$ & 2\;0 & 2\;0 & $9_{32}$ &  &  &          $10_{33}$ &  &  &           $10_{83}$ &  &  &         $10_{133}$ & 2 &  \\
$8_4$ &  & 2\;0 &     $9_{33}$ &  &  &          $10_{34}$ &  &  &           $10_{84}$ &  &  &         $10_{134}$ & 1 &  \\
$8_5$ & 2 &  &        $9_{34}$ &  &  &          $10_{35}$ &  &  &           $10_{85}$ &  &  &         $10_{135}$ &  &  \\
$8_6$ & 2 &  &        $9_{35}$ & 1 &  &         $10_{36}$ & 2\;0 &  &       $10_{86}$ &  &  &         $10_{136}$ &  & 0\;0 \\
$8_7$ &  &  &         $9_{36}$ &  & 2\;0 &      $10_{37}$ &  &  &           $10_{87}$ &  &  &         $10_{137}$ &  &  \\
$8_8$ &  &  &         $9_{37}$ &  &  &          $10_{38}$ & 2\;0 &  &       $10_{88}$ &  &  &         $10_{138}$ &  &  \\
$8_9$ &  &  &         $9_{38}$ & 1 &  &         $10_{39}$ & 2\;0 &  &       $10_{89}$ &  &  &         $10_{139}$ & 1 &  \\
$8_{10}$ &  &  &      $9_{39}$ &  & 2\;0\;0 &   $10_{40}$ &  &  &           $10_{90}$ &  &  &         $10_{140}$ &  &  \\
$8_{11}$ & 2\;0 &  &  $9_{40}$ &  &  &          $10_{41}$ &  &  &           $10_{91}$ &  &  &         $10_{141}$ &  &  \\
$8_{12}$ &  &  &      $9_{41}$ &  &  &          $10_{42}$ &  &  &           $10_{92}$ & 2\;0\;0 &  &  $10_{142}$ & 1 &  \\
$8_{13}$ &  &  &      $9_{42}$ &  & 0 &         $10_{43}$ &  &  &           $10_{93}$ &  &  &         $10_{143}$ &  &  \\
$8_{14}$ & 2\;0 &  &  $9_{43}$ & 2 &  &         $10_{44}$ &  &  &           $10_{94}$ &  &  &         $10_{144}$ & 2\;0 &  \\
$8_{15}$ & 1 &  &     $9_{44}$ &  &  &          $10_{45}$ &  &  &           $10_{95}$ &  &  &         $10_{145}$ &  & 1 \\
$8_{16}$ &  &  &      $9_{45}$ &  & 2 &         $10_{46}$ & 2 &  &          $10_{96}$ &  &  &         $10_{146}$ &  &  \\
$8_{17}$ &  &  &      $9_{46}$ & 2\;0 & 2 &     $10_{47}$ &  &  &           $10_{97}$ & 2\;0\;0 &  &  $10_{147}$ & 2\;0\;0 &  \\
$8_{18}$ &  &  &      $9_{47}$ &  &  &          $10_{48}$ &  &  &           $10_{98}$ & 2\;0\;0 &  &  $10_{148}$ &  &  \\
$8_{19}$ & 1 &  &     $9_{48}$ &  & 2\;0\;0 &   $10_{49}$ & 1 &  &          $10_{99}$ &  &  &         $10_{149}$ & 2 &  \\
$8_{20}$ &  &  &      $9_{49}$ & 1 &  &         $10_{50}$ & 2\;0 &  &       $10_{100}$ &  &  &        $10_{150}$ & 2\;0 &  \\
$8_{21}$ & 2 &  &     $10_1$ & 2 & 2\;0 &       $10_{51}$ &  &  &           $10_{101}$ & 1 &  &       $10_{151}$ &  &  \\
$9_1$ & 1 &  &        $10_2$ & 2 &  &           $10_{52}$ &  &  &           $10_{102}$ &  &  &        $10_{152}$ & 1 &  \\
$9_2$ & 1 &  &        $10_3$ & 2\;0 & 2\;0\;0 & $10_{53}$ & 1 &  &          $10_{103}$ &  &  &        $10_{153}$ &  &  \\
$9_3$ & 1 &  &        $10_4$ &  & 2\;0\;0 &     $10_{54}$ &  &  &           $10_{104}$ &  &  &        $10_{154}$ & 1 &  \\
$9_4$ & 1 &  &        $10_5$ &  &  &            $10_{55}$ & 1 &  &          $10_{105}$ &  &  &        $10_{155}$ &  &  \\
$9_5$ & 1 &  &        $10_6$ & 2 &  &           $10_{56}$ & 2\;0 &  &       $10_{106}$ &  &  &        $10_{156}$ &  &  \\
$9_6$ & 1 &  &        $10_7$ & 2\;0 &  &        $10_{57}$ &  &  &           $10_{107}$ &  &  &        $10_{157}$ &  & 2 \\
$9_7$ & 1 &  &        $10_8$ & 2\;0 &  &        $10_{58}$ &  &  &           $10_{108}$ &  &  &        $10_{158}$ &  &  \\
$9_8$ & 2\;0\;0 &  &  $10_9$ &  &  &            $10_{59}$ &  &  &           $10_{109}$ &  &  &        $10_{159}$ &  &  \\
$9_9$ & 1 &  &        $10_{10}$ &  &  &         $10_{60}$ &  &  &           $10_{110}$ &  &  &        $10_{160}$ & 2\;0 &  \\
$9_{10}$ & 1 &  &     $10_{11}$ & 2\;0\;0 &  &  $10_{61}$ & 2\;0 &  &       $10_{111}$ & 2\;0\;0 &  & $10_{161}$ & 1 &  \\
$9_{11}$ &  & 2\;0 &  $10_{12}$ &  &  &         $10_{62}$ &  &  &           $10_{112}$ &  &  &        $10_{162}$ &  & 2\;0 \\
$9_{12}$ & 2\;0 &  &  $10_{13}$ &  &  &         $10_{63}$ & 1 &  &          $10_{113}$ &  &  &        $10_{163}$ &  &  \\
$9_{13}$ & 1 &  &     $10_{14}$ & 2\;0 &  &     $10_{64}$ &  &  &           $10_{114}$ &  &  &        $10_{164}$ &  &  \\
$9_{14}$ &  &  &      $10_{15}$ &  &  &         $10_{65}$ &  &  &           $10_{115}$ &  &  &        $10_{165}$ &  & 2 \\
$9_{15}$ &  & 2\;0 &  $10_{16}$ & 2\;0\;0 &  &  $10_{66}$ & 1 &  &          $10_{116}$ &  &  &
\end{tabular}
\end{small}
\end{center}
\caption{Khovanov cobordism maps induced by the unique incompressible Seifert surfaces $\S\colon L\to\emptyset$ and their mirror $-\S\colon{-L}\to\emptyset$.}\label{table-Seifert}
\end{table}
}

For non-alternating knots, Seifert's algorithm often still yields minimal genus Seifert surfaces. Only 18 knots out of 249 required additional adjustments. For these cases, we repeatedly perturbed the diagram using Reidemeister moves (occasionally increasing the crossing number by 1 or 2) until we successfully obtained a minimal genus Seifert surface using Seifert's algorithm in every instance. The PD codes and minimal genus data were sourced from KnotInfo \cite{knotinfo}.

Using this data of Seifert surfaces, we computed the cobordism maps in matrix form for both $\S$ and $-\S$, the mirror of the surface (bounded by the mirror of the knot). \Cref{table-Seifert} shows the resulting matrices, where an empty cell means $k=0$, and an entry such as $2\;0\;0$ indicates $k=3$.

Viewing the cobordism $\S$ in the opposite direction, denoted $\S^\star$, we obtain a map $\KH^{0,0}(\emptyset)\to\KH^{0,\chi}(L)$, which is determined by the Khovanov--Jacobsson class of $\S^\star$. However, these maps $\KH(\S^\star)$ provide no new information, since $\CKH(-\S)$ and $\CKH(\S^\star)^\star$ are canonically related as described in \cref{sec:mirroring}.

There have been studies on the Khovanov--Jacobsson class of surfaces obtained via Seifert's algorithm \cites{Sundberg-Swann, Elliott1, Elliott2}, but only on the chain level rather than homology. These studies help interpret some of the results in \cref{table-Seifert}, using the state graph $\Gamma_\sigma$ of the Seifert smoothing $\sigma$ of a diagram. This graph's vertices represent the loops of $D_\sigma$, while its edges correspond to the crossings of the link diagram connecting two loops. The edges are coloured red or blue based on whether the crossing is positive or negative (i.e. 0-smoothed or 1-smoothed in $\sigma$). The Khovanov--Jacobsson chain element's structure only depends on the Betti numbers $b_0,b_1$ of the red subgraph of $\Gamma_\sigma$ \cite{Sundberg-Swann}.

		\subsection{Ribbon disks from symmetries}
Let $L$ be a ribbon knot with ribbon disk $\S$ and suppose that $L$ has a non-trivial symmetry: a diffeomorphism $\varphi\colon S^3\to S^3$ such that $\varphi(L)=L$ (disregarding orientation of $L$) and $\varphi$ is not isotopic to the identity through diffeomorphisms that fix $L$ setwise. Viewing $\S$ as an immersed surface in $S^3$ with ribbon singularities, consider $\varphi(\S)$. This is also a ribbon disk for $L$, which often differs from $\S$ (up to smooth isotopy relative to the boundary in $D^4$). For instance, the two exotically knotted ribbon disks in \cite{Hayden-Sundberg}*{Figure 1} and the annuli in \cite{Feher-pqkn}*{Theorem 1} are both related by symmetry and distinguished smoothly by Khovanov homology.

In this section, we study ribbon disks obtained via symmetries for all prime ribbon knots with at most 9 crossings and for certain 10-crossing ribbon knots, using Khovanov homology.

The group of diffeomorphisms of the pair $(S^3,L)$ up to isotopy is called the \emph{symmetry group} of $L$. If $L$ is hyperbolic, this group is finite, and can be computed from the hyperbolic triangulation of the knot complement. KnotInfo \cite{knotinfo} provides data on these symmetry groups, summarized in \cref{table-ribbon}. Here, $D_n$ denotes the dihedral group of $2n$ elements, corresponding to the symmetry group of a regular $n$-gon. The table also lists the number of different ribbon disks we identified, distinguished by Khovanov homology.

\begin{table}[h]
\begin{center}
\renewcommand{\arraystretch}{1.5}
\begin{tabular}{r|ccccccccc}
Knot               & $6_1$ & $8_8$ & $8_9$ & $8_{20}$ & $9_{27}$ & $9_{41}$ & $9_{46}$ & $10_3$ & $10_{123}$ \\
\hline
Symmetry group     & $D_2$ & $D_2$ & $D_4$ & $D_1$    & $D_2$    & $D_3$    & $D_2$    & $D_2$  & $D_{10}$ \\
Ribbon disks found & 2     & 2     & 4     & 1        & 2        & 2        & 2        & 2      & 5 
\end{tabular}
\renewcommand{\arraystretch}{1.0}
\end{center}
\caption{Ribbon disks distinguished by Khovanov homology.}\label{table-ribbon}
\end{table}

Knowing the symmetry group of a knot does not always make it straightforward to identify its symmetries. A knot $L$ is called \emph{reversible}, if there exists a symmetry $\varphi$ that reverses the orientation of $L$. For any reversible prime knot with at most 10 crossings, there exists a \emph{symmetric diagram}, where an orientation-reversing symmetry $\varphi$ is realized as a $180^\circ$ rotation, as listed in \cites{sakuma,lamm}. Every ribbon knot up to 10 crossings that admits a non-trivial symmetry is reversible \cite{knotinfo}.

\begin{figure}[!p]
\centering
\par{\centering\def\svgwidth{0.95\textwidth}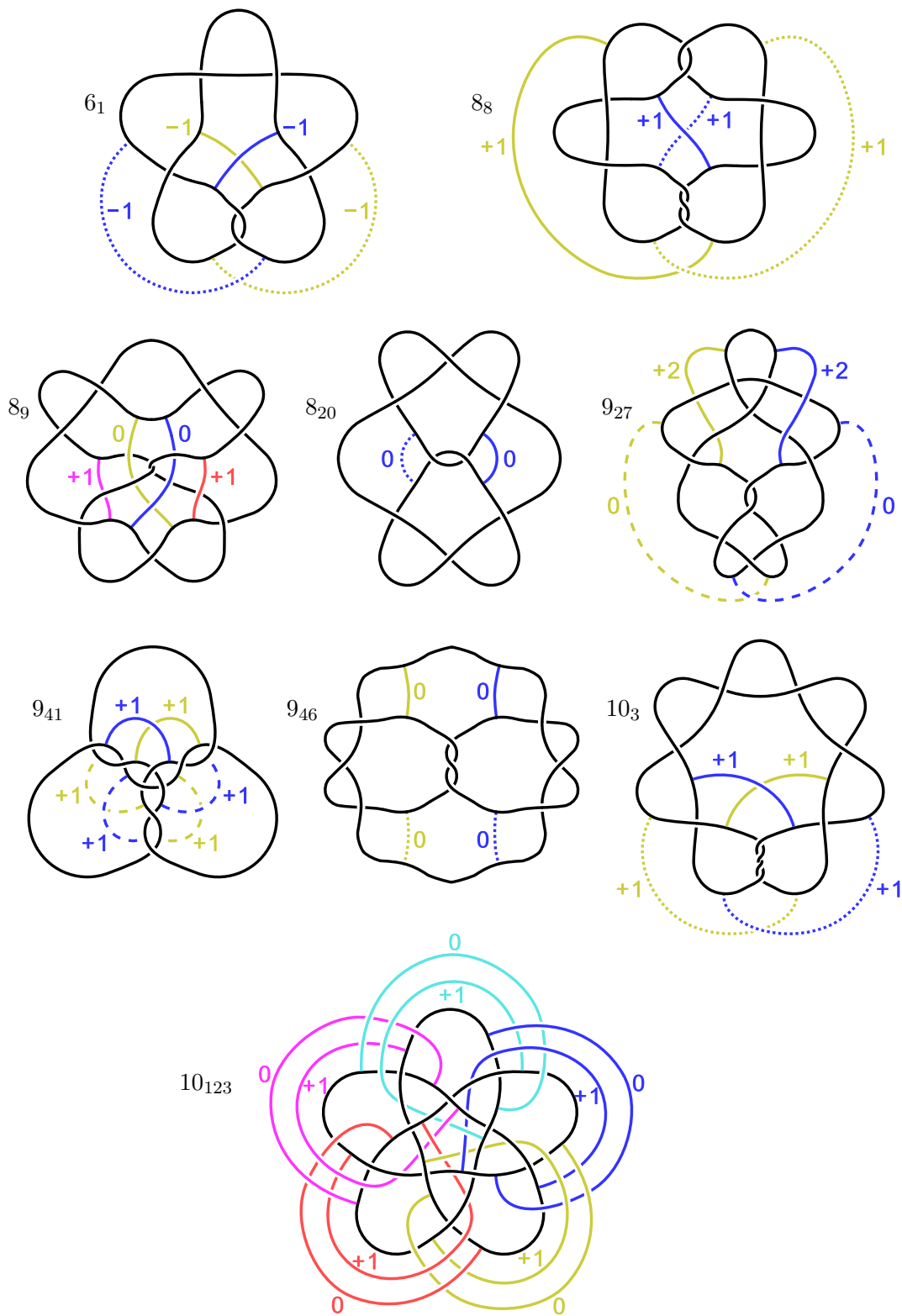\par}
\caption{Band diagrams for ribbon disks obtained by symmetry. Solid lines of different colours represent non-isotopic disks distinguished by Khovanov homology.}
\label{fig:ribbon}
\end{figure}

Eisermann--Lamm \cite{Eisermann-Lamm} provide band diagrams for one ribbon disk for each ribbon knot up to 10 crossings (given as symmetric union diagrams). Using KLO \cite{KLO}, we manually transformed these into symmetric diagrams to find additional ribbon disks derived from the symmetries. In certain cases ($8_9$ and $9_{27}$), further manual manipulations were performed to search for the remaining symmetries. These ribbon disks were then compared using our Khovanov homology program. The source code for these computations is provided in the ancillary file \mintinline{text}{ribbon_disks.py} and is also available on GitHub \cite{github}.

\Cref{fig:ribbon} shows the band diagrams for ribbon disks obtained via symmetries, layered over the same knot diagram. With the exception of $10_{123}$, each band diagram consists of a single band. Solid lines represent ribbon disks that are distinguished by Khovanov homology. Dotted lines indicate disks that are isotopic to others, while dashed lines represent disks that cannot be distinguished by Khovanov homology (both these disks and their mirrors share the same Khovanov--Jacobsson class with another disk in the figure). Disks sharing the same Khovanov homology are shown in the same colour.

\begin{itemize}
\item $6_1$, $8_8$, $9_{46}$, $10_3$: Likely all four symmetries were identified, yielding two distinct ribbon disks. The other two disks are equivalent to the first two via band isotopy.
\item $8_9$: Four of eight symmetries were found, resulting in four distinct ribbon disks.
\item $8_{20}$: Both symmetries were identified, but they produce the same ribbon disk.
\item $9_{27}$: Likely all four symmetries were found, yielding at least two distinct ribbon disks. The other two are not distinguished by Khovanov homology but are not evidently the same band diagram.
\item $9_{41}$: Likely all six symmetries were identified, resulting in at least two distinct ribbon disks. The other four are not distinguished by Khovanov homology.
\item $10_{123}$: Chosen for its high number of symmetries, likely all 20 were identified, yielding at least five distinct ribbon disks. The remaining disks are not distinguished by Khovanov homology.
\end{itemize}
Some of these ribbon disks have already been observed to be non-isotopic. For instance, see \cite{Fox}*{Examples 10--11} for $6_1$ and \cites{Miller, Sundberg-Swann} for $9_{46}$.

\end{document}